\documentclass[11pt,reqno]{amsart}
\usepackage[a4paper,margin=28mm]{geometry}
\usepackage[T1]{fontenc}
\usepackage{lmodern}
\usepackage{amsmath,amssymb,amsthm,mathtools}
\usepackage[expansion=false]{microtype}
\usepackage[colorlinks=true,linkcolor=blue,citecolor=blue,urlcolor=blue]{hyperref}

\newtheorem{theorem}{Theorem}[section]
\newtheorem{lemma}[theorem]{Lemma}
\newtheorem{proposition}[theorem]{Proposition}
\newtheorem{corollary}[theorem]{Corollary}
\theoremstyle{remark}

\numberwithin{equation}{section}

\newcommand{\R}{\mathbb{R}}
\newcommand{\C}{\mathbb{C}}
\newcommand{\Z}{\mathbb{Z}}
\newcommand{\mcS}{\mathcal S}
\newcommand{\md}[1]{\mathrm{d}#1}
\newcommand{\Fs}{\mathcal F_s}
\newcommand{\B}{\mathcal B}
\newcommand{\OpBJ}{\operatorname{Op}_{BJ}}

\title[Endpoint bounds for the Born-Jordan distribution]
{Endpoint bounds for the Born-Jordan distribution}

\author{Yaojun Wang}
\address[Yaojun Wang ] {Mathematisches Institut, Ludwig-Maximilians-Universit\"at M\"unchen,
	Theresienstr. 39, 80333 M\"unchen, Germany}
\email{yaojun.wang@math.lmu.de}

\date{}
\subjclass[2020]{Primary 42B10; Secondary 42B35, 47G30}
\keywords{Born-Jordan distribution, endpoint estimate, rearrangement inequality,
Gaussian functions, complex interpolation}

\begin{document}

\begin{abstract}
	For $d\geqslant3$, let $p^*=\frac{2d}{d-2}$. We prove that the
	Born-Jordan distribution maps $L^{2,p^*}(\R^d)\times L^{2,p^*}(\R^d)$
	boundedly into $L^{p^*}(\R^{2d})$. The Lorentz index $p^*$ cannot
	be increased in either variable. In particular, this gives the
	endpoint $L^2\times L^2\to L^{p^*}$ estimate, settling the critical boundedness problem raised by Stra, Svela, and Trapasso \cite{SST26}.
\end{abstract}

\maketitle

\section{Introduction}\label{sec:introduction}

A signal with finite energy is modelled by a function $f\in L^2(\R^d)$.
Its Fourier transform describes its frequency content, but does not
directly show where these frequencies occur in time or space. In signal
processing, it is therefore useful to associate with $f$ a representation
on the phase space $\R^d_x\times\R^d_\xi$, which describes position and
frequency together; see \cite{Coh89,HB92}. Once such a representation is
chosen, one can ask about its regularity and integrability, and how
strongly it can concentrate in a given region of phase space. We are
interested here in integrability bounds in terms of the signal energy.

Different phase-space representations have different advantages and
disadvantages. The spectrogram, which is the squared modulus of the
short-time Fourier transform, is nonnegative, but its time and frequency
resolution depends on the chosen window \cite{HB92}. The Wigner distribution does
not require a window and has the correct position and frequency
marginals for sufficiently regular signals. However, it may take
negative values, and superpositions of signals produce oscillatory
cross-terms \cite{Coh89}. Cohen's class \cite{Coh66} consists of quadratic
representations obtained by convolving the Wigner distribution with a
kernel. The choice of kernel affects the interference patterns and the
analytic properties of the representation.

In this paper we study the Born-Jordan distribution. The related
quantization rule was introduced by Born and Jordan in 1925
\cite{BJ25}. The associated phase-space distribution appeared in
Cohen's framework \cite{Coh66}. In time-frequency
analysis, Boggiatto, De Donno, and Oliaro \cite{BDO10} studied the
average of the $\tau$-Wigner distributions over $[0,1]$, which is the Born-Jordan distribution. It retains the position and frequency
marginals and reduces certain interference patterns. Cordero,
de Gosson, and Nicola \cite{CDN18} proved that the smoothing depends
on the direction in phase space. For a survey of Born-Jordan-type
distributions and their use in signal analysis, see \cite{CDGN21}.

In the setting of modulation spaces, Cordero and Nicola \cite{CN18}
obtained sharp bounds for the cross-Wigner distribution, with
applications to Cohen's class, including the Born-Jordan distribution.
Cordero, de Gosson, and Nicola \cite{CDN17} also established
boundedness results for Born-Jordan pseudodifferential operators
on modulation spaces.
We focus here on global $L^p$ estimates for the Born-Jordan distribution.

For $f,g\in\mcS(\R^d)$ and $0\leqslant\tau\leqslant1$, the
cross-$\tau$-Wigner distribution is defined by
\begin{equation}\label{eq:wtau}
 W_\tau(f,g)(x,\xi)
 =\int_{\R^d}e^{-2\pi i\xi\cdot y}
 f(x+\tau y)\overline{g(x-(1-\tau)y)}\md{y}.
\end{equation}
The choice $\tau=\frac{1}{2}$ recovers the usual cross-Wigner
distribution. Averaging over the parameter $\tau$ yields the Born-Jordan distribution,
\begin{equation*}
 W_{BJ}(f,g)=\int_0^1W_\tau(f,g)\md{\tau}.
\end{equation*}
For general $f,g\in L^2(\R^d)$, this average is well-defined as an
$L^2(\R^{2d})$-valued integral. Moyal's identity
$\|W_\tau(f,g)\|_2=\|f\|_2\|g\|_2$ then yields the basic estimate
$\|W_{BJ}(f,g)\|_2\leqslant\|f\|_2\|g\|_2$.

Lieb's sharp inequality \cite{Lie90} provides the starting point for
$L^p$ estimates. After a change of variables, it yields
\begin{equation}\label{eq:lieb-tau}
 \|W_\tau(f,g)\|_p
 \leqslant\left(\frac{2}{p}\right)^{\frac{d}{p}}
 [\tau(1-\tau)]^{d(\frac{1}{p}-\frac{1}{2})}
 \|f\|_2\|g\|_2,
 \qquad 2\leqslant p<\infty,\quad 0<\tau<1.
\end{equation}
The constant in \eqref{eq:lieb-tau} is optimal; for $2<p<\infty$ and
each fixed $\tau\in(0,1)$, equality is attained by suitable Gaussian
pairs.
For $d\geqslant3$, combining \eqref{eq:lieb-tau} with Minkowski's
inequality gives the global $L^2\times L^2\to L^p$ bound in the
subcritical range $2\leqslant p<p^*=\frac{2d}{d-2}$, where the
parameter weight is integrable. At the critical exponent $p=p^*$,
this weight becomes $[\tau(1-\tau)]^{-1}$, and the resulting upper
bound diverges logarithmically at both endpoints of the parameter
interval. Thus this direct argument does not reach the endpoint.

The same critical exponent appears in the concentration problem
studied by Stra, Svela, and Trapasso \cite{SST26}, namely
\begin{equation*}
 \sup_{\|f\|_2=1}\|W_{BJ}(f,f)\|_{L^p(\Omega)},
 \qquad \Omega\subset\R^{2d},\quad 0<|\Omega|<\infty.
\end{equation*}
For $d\geqslant3$, they established finiteness and attainment when
$1\leqslant p<p^*$, and unboundedness when $p>p^*$.
The critical case $p=p^*$ was left open; in fact, the authors suggested
that the supremum should also be infinite at this threshold
\cite[Section 1.1, after Theorem 1.7]{SST26}.
We prove that it is finite by establishing a strong-type
estimate at $p=p^*$. In fact, the estimate extends from $L^2(\R^d)$
to the Lorentz space $L^{2,p^*}(\R^d)$ in each variable, as follows.

\begin{theorem}[Endpoint Lorentz estimate]\label{thm:main}
	Let $d\geqslant3$ and $p^*=\frac{2d}{d-2}$. There exists a constant $C_d$ depending only on $d$ such that for every $f,g\in L^{2,p^*}(\R^d)$,
	\begin{equation}\label{eq:lorentz-endpoint}
		\|W_{BJ}(f,g)\|_{L^{p^*}(\R^{2d})}
		\leqslant C_d\|f\|_{2,p^*}\|g\|_{2,p^*}.
	\end{equation}
	The Lorentz index $p^*$ cannot be increased in either variable.
	
	In particular, for every $f,g\in L^2(\R^d)$, we have the endpoint
	estimate
	\begin{equation*}
		\|W_{BJ}(f,g)\|_{p^*}\leqslant C_d\|f\|_2\|g\|_2.
	\end{equation*}
\end{theorem}

Theorem \ref{thm:main} settles the critical boundedness problem with a
constant independent of $\Omega$. Together with the supercritical
examples in \cite{SST26}, it identifies $p^*$ as the largest exponent
$p\geqslant2$ for which a global $L^2\times L^2\to L^p$ estimate can
hold.

Theorem \ref{thm:main} also yields an endpoint estimate for
Born-Jordan pseudodifferential operators. Recall that, for a symbol $a\in L^r(\R^{2d})$,
$1\leqslant r\leqslant\infty$, the Born-Jordan operator is
defined weakly by (see \cite[Section 6]{BDO10} and \cite[(4)]{CDN17})
\begin{equation*}
	\langle\OpBJ(a)f,g\rangle
	:=\int_{\R^{2d}}a(X)\overline{W_{BJ}(g,f)(X)}\md{X},
	\qquad f,g\in\mcS(\R^d).
\end{equation*}
For $d\geqslant3$, Boggiatto, De Donno, and Oliaro \cite[Theorem 6.10]{BDO10} proved $L^2$-boundedness for symbols
in $L^r(\R^{2d})$ when $\frac{2d}{d+2}<r\leqslant2$.
The following corollary treats the left endpoint, which was left open
in \cite[Appendix B]{SST26}.

\begin{corollary}\label{cor:bj-operator}
	Let $d\geqslant3$ and $q^*=(p^*)'=\frac{2d}{d+2}$.
	For every $a\in L^{q^*}(\R^{2d})$, the operator $\OpBJ(a)$
	has a unique bounded extension to $L^2(\R^d)$, with
	\begin{equation*}
		\|\OpBJ(a)\|_{\B(L^2(\R^d))}
		\leqslant C_d\|a\|_{L^{q^*}(\R^{2d})},
	\end{equation*}
	where $C_d$ depends only on $d$.
\end{corollary}

We now outline the proof. The symmetry relation
$W_{1-\tau}(f,g)=\overline{W_\tau(g,f)}$ reduces matters to the average
over $(0,\frac{1}{2})$. The key ingredient is a weighted even-moment
estimate: for $d\geqslant3$ and every even integer $q\geqslant4$, with
$\alpha=d(\frac{1}{2}-\frac{1}{q})$, we establish, for
$f,g\in L^{2,q}(\R^d)$,
\begin{equation*}
	\left\|\int_0^{\frac{1}{2}}v(\tau)W_\tau(f,g)\md{\tau}\right\|_q
	\leqslant C_{d,q}M\|f\|_{2,q}\|g\|_{2,q},
	\qquad\text{if } |v(\tau)|\leqslant M\tau^{\alpha-1}.
\end{equation*}
Here $v$ may be complex-valued, and the average is understood by
continuous extension as in Theorem \ref{thm:weighted}. An application of Minkowski's inequality using \eqref{eq:lieb-tau}
and the weight bound would again produce a logarithmically divergent
upper bound. To overcome this obstruction, we expand a regularized
even moment of the entire average and apply the
Brascamp-Lieb-Luttinger inequality \cite{BLL74}. For bounded,
compactly supported functions, this reduces the estimate to symmetric
decreasing rearrangements, which we majorize by finite positive sums
of Gaussians. The coefficients of each majorant are controlled in
$\ell^q$ by the corresponding $L^{2,q}$ quasi-norm. A Gaussian bound and dyadic summation yield the weighted estimate.

To obtain the endpoint estimate for all $d\geqslant3$, we embed the
average over $(0,\frac{1}{2})$ in an analytic family.
A Fourier-multiplier bound gives the estimate
$L^2\times L^2\to L^2$ on the left boundary, while the weighted
eighth-moment estimate gives
$L^{2,8}\times L^{2,8}\to L^8$ on the right boundary.
Applying the bilinear version of Stein's complex interpolation
theorem \cite{Ste56,SW71,GM14} then yields the endpoint Lorentz
estimate $L^{2,p^*}\times L^{2,p^*}\to L^{p^*}$.
The sharpness of the Lorentz indices follows by testing on
truncated radial power functions; see Proposition~\ref{prop:lorentz-sharpness}.

The paper is organized as follows.
In Section \ref{sec:preparations}, we collect the preliminary results
and prove the sharpness of the Lorentz indices.
We prove the weighted even-moment estimate in
Section \ref{sec:even-moment} and complete the proofs of
Theorem \ref{thm:main} and Corollary~\ref{cor:bj-operator}
in Section~\ref{sec:proof}.

\noindent\textbf{Notation.}
We write $\|h\|_{L^p(\R^d)}=\|h\|_p$ when the underlying space is
clear. For $0<p<\infty$, $0<q\leqslant\infty$, we denote the Lorentz quasi-norm of $h$ by $\|h\|_{p,q}$. We denote the volume of the unit ball in $\R^d$ by $\omega_{d}$. The constants $c_d, C_d, C_{d,q}$ depend only on the indicated
parameters and may vary from line to line. 

For a measurable function $h$ with $|\{|h|>t\}|<\infty$
for all $t>0$, let $h^*$ denote the symmetric decreasing
rearrangement of $|h|$. It is nonnegative, radial and
nonincreasing, and satisfies
\begin{equation*}
	|\{h^*>t\}|=|\{|h|>t\}|,\qquad t>0.
\end{equation*}
In particular, $\|h^*\|_p=\|h\|_p$ for $1\leqslant p<\infty$. See \cite[Chapter 3]{LL01} for more details.

\bigskip
\begingroup
\emergencystretch=2em
\noindent\textbf{Acknowledgements.}\quad
The author acknowledges partial support from the China Scholarship Council
and Ludwig-Maximilians-Universit\"at M\"unchen. The author is grateful to his
advisor, Professor Rupert Frank, for suggesting this problem and for
valuable discussions.

\medskip
\noindent\textbf{Declaration of AI Use.}\quad
At an early stage of this work, the author established the bounds
$W_{BJ}:L^2(\R^d)\times L^2(\R^d)\to L^{p^*,s}(\R^{2d})$
for every $s>p^*$ by complex interpolation, without using AI tools.
In subsequent discussions with ChatGPT, the author observed that
expanding the even-order moments in the right-boundary estimate
allows an application of the Brascamp-Lieb-Luttinger inequality.
AI tools were also used extensively to polish the exposition, correct
grammatical errors, and improve the wording, readability, and formatting
of the manuscript. All mathematical arguments and proofs in the final manuscript were checked and written by the author.

\par\endgroup

\section{Preparations}\label{sec:preparations}

\subsection{Fourier identities and standard inequalities}

We use the following convention for the Fourier transform:
$$
\widehat {h}(\xi)=\int_{\R^d}e^{-2\pi i x\cdot\xi}h(x)\md{x}.
$$
In particular, for Gaussians we have 
\begin{equation}\label{eq:gaussian-fourier}
 \int_{\R^d}e^{-\pi t|y|^2}e^{-2\pi i y\cdot\xi}\md{y}
 =t^{-\frac{d}{2}}e^{-\pi\frac{|\xi|^2}{t}},\qquad t>0.
\end{equation}
See for example \cite[Example 2.2.9 and Proposition 2.2.11(8)]{Gra14}.

Let $F\in L^2(\R^{2d})$. The symplectic Fourier transform $F\mapsto \Fs F$ given by
\begin{equation*}
 \Fs F(y,\eta):=\iint_{\R^{2d}}
 e^{-2\pi i(\xi\cdot y-x\cdot\eta)}F(x,\xi)\md{x}\md{\xi}
 =\widehat F(-\eta,y)
\end{equation*}
is a unitary operator on $L^2(\R^{2d})$. 

For $f,g\in L^2(\R^d)$, set $A(f,g)=\Fs W_{1/2}(f,g)$. A direct computation shows
\begin{equation}\label{eq:phase}
	\Fs W_\tau(f,g)(y,\eta)
	=e^{2\pi i(\tau-\frac{1}{2})y\cdot\eta}A(f,g)(y,\eta),
\end{equation}
where the identity holds in $L^2(\R^{2d})$. 

The following facts are standard and we give a proof for completeness.

\begin{lemma}\label{lem:moyal}
Let $f,g\in L^2(\R^d)$. Then we have
\begin{enumerate}
	\item  The symmetry relation (see \cite[Proposition 1.3.23(i)]{CR20}):
	\begin{equation}\label{eq:reflection}
		W_{1-\tau}(f,g)=\overline{W_\tau(g,f)}.
	\end{equation}
	\item Moyal's identity (see \cite[Corollary 1.3.28]{CR20}):
	\begin{equation}\label{eq:moyal}
		\|W_\tau(f,g)\|_2=\|A(f,g)\|_2=\|f\|_2\|g\|_2.
	\end{equation}
	\item  For any $v\in L^1(0,1)$, the map
	$$
	(f,g)\in L^2(\R^d)\times L^2(\R^d)\mapsto\int_0^1v(\tau)W_\tau(f,g)\md{\tau}
	$$
	is well defined and satisfies
	$$
	\left\|\int_0^1 v(\tau)W_{\tau}(f,g)\md{\tau}\right\|_{2}\leqslant \| v\|_1 \|f\|_2\|g\|_2.
	$$	
\end{enumerate}
\end{lemma}

\begin{proof}
	(1) By definition, we have
 \begin{align*}
 	W_{1-\tau}(f,g)&=\int_{\R^d}e^{-2\pi i \xi\cdot y}f(x+(1-\tau)y)\overline{g(x-\tau y)}\md y\\
 	&=\overline{\int_{\R^d} e^{-2\pi i\xi \cdot y} g(x+\tau y)\overline{f(x-(1-\tau)y)}\md y}\\
 	&=\overline{W_\tau(g,f)}.
 \end{align*}

(2)	 The first equality in \eqref{eq:moyal} is due to \eqref{eq:phase} and Plancherel's theorem. By Plancherel, we also have
\begin{align*}
	\|W_\tau(f,g)\|_2^2&=\|(f(x+\tau \cdot)\overline{g(x-(1-\tau)\cdot)})^\wedge(\xi)\|_2^2\\
	&=\int_{\R^{2d}} |f(x+\tau y)|^2|g(x-(1-\tau)y)|^2\md x\md y.
\end{align*}
Using a change of variables $u=x+\tau y$ and $v=x-(1-\tau)y$, we obtain
$$
\|W_\tau(f,g)\|_2^2=\int_{\R^{2d}} |f(u)|^2|g(v)|^2\md u\md v=\|f\|_2^2\|g\|_2^2.
$$

(3) Let $\tau_n, \tau\in [0,1]$ and $\tau_n\to\tau$ as $n\to \infty$. Then for every $(y,\eta)\in \R^{2d}$,
$$
\left| e^{2\pi i (\tau_n-\frac{1}{2})y\cdot \eta}-e^{2\pi i(\tau-\frac{1}{2})y\cdot \eta}\right|\to 0,\qquad \text{as }n\to \infty.
$$
By \eqref{eq:phase}, one has
$$
\| W_\tau(f,g)-W_{\tau_n}(f,g)\|_2^2=\int_{\R^{2d}} \left| e^{2\pi i (\tau_n-\frac{1}{2})y\cdot \eta}-e^{2\pi i(\tau-\frac{1}{2})y\cdot \eta}\right|^2 |A(f,g)(y,\eta)|^2\md y\md\eta.
$$
Since $\|A(f,g)\|_2=\|f\|_2\|g\|_2$, the dominated convergence theorem gives
$$
\lim_{n\to\infty}\| W_\tau(f,g)-W_{\tau_n}(f,g)\|_2^2=0,
$$
which implies $\tau\mapsto W_\tau(f,g)$ is continuous. Therefore the map is strongly measurable and well-defined. For every $v\in L^1(0,1)$, Minkowski's inequality and Moyal's identity \eqref{eq:moyal} give
\begin{equation*}
 \left\|\int_0^1 v(\tau)W_{\tau}(f,g)\md{\tau}\right\|_{2}\leqslant\int_0^1\|v(\tau)W_\tau(f,g)\|_2\md{\tau}
 =\|v\|_1\|f\|_2\|g\|_2<\infty.
\end{equation*}
\end{proof}

We shall use the following Brascamp-Lieb-Luttinger inequality later. See \cite[Theorem 3.8]{LL01} and the original paper \cite{BLL74}.
\begin{lemma}[Brascamp-Lieb-Luttinger]
\label{lem:bll}
Let $f_1, f_2,\dots, f_m$ be nonnegative functions on $\R^d$, whose positive superlevel sets have finite measure. Let $k\leqslant m$ and let $B=\{b_{ij}\}$ be a $k\times m$ matrix with $(1\leqslant i\leqslant k, 1\leqslant j\leqslant m)$. Define
$$
I(f_1,\dots,f_m):=\int_{\R^d}\cdots\int_{\R^d} \prod_{j=1}^m f_j\left(\sum_{i=1}^k b_{ij}x_i\right)\md x_1\cdots \md x_k.
$$
Then $I(f_1,\dots, f_m)\leqslant I(f_1^*,\dots, f_m^*)$.
\end{lemma}

We shall use the following bilinear version of Stein's complex
interpolation theorem \cite{Ste56}, which follows from the result
of Grafakos and Masty{\l}o \cite[Theorem 1.1]{GM14}.

\begin{lemma}\label{lem:stein}
	Let $2<q<\infty$. For $0\leqslant\mathrm{Re}z\leqslant1$, let
	$\mathcal T_z$ be a sesquilinear family on
	$\mcS(\R^d)\times\mcS(\R^d)$ with values in $L^2(\R^{2d})$.
	Suppose that, for every fixed $f,g\in\mcS(\R^d)$, the map $z\mapsto\mathcal T_z(f,g)$ is
	$L^2(\R^{2d})$-valued, strongly holomorphic in the open strip $\{z\in \C: 0<\mathrm{Re}z<1\}$, and continuous on its closure. Assume that, for some $N\geqslant0$,
	\begin{equation}\label{eq:lorentz-interpolation-growth}
		\|\mathcal T_{\sigma+it}(f,g)\|_2
		\leqslant C_{f,g}(1+|t|)^N,
		\qquad 0\leqslant\sigma\leqslant1,\quad t\in\R.
	\end{equation}
	Suppose that there exist $A_0,A_1>0$ and $N_0,N_1\geqslant0$ such
	that, for every $f,g\in\mcS(\R^d)$ and $t\in\R$,
	\begin{equation}\label{eq:lorentz-left}
		\|\mathcal T_{it}(f,g)\|_2
		\leqslant A_0(1+|t|)^{N_0}\|f\|_2\|g\|_2
	\end{equation}
	and
	\begin{equation}\label{eq:lorentz-right}
		\|\mathcal T_{1+it}(f,g)\|_q
		\leqslant A_1(1+|t|)^{N_1}\|f\|_{2,q}\|g\|_{2,q}.
	\end{equation}
	If $0<\theta<1$ and
	\begin{equation*}
		\frac{1}{p}=\frac{1-\theta}{2}+\frac{\theta}{q},
	\end{equation*}
	then
	\begin{equation}\label{eq:lorentz-interpolation-conclusion}
		\|\mathcal T_\theta(f,g)\|_p
		\leqslant C_{q,\theta,N_0,N_1}A_0^{1-\theta}A_1^\theta
		\|f\|_{2,p}\|g\|_{2,p}
	\end{equation}
	for every $f,g\in\mcS(\R^d)$. Consequently, $\mathcal T_\theta$
	has a unique continuous sesquilinear extension
	\begin{equation*}
		\mathcal T_\theta:
		L^{2,p}(\R^d)\times L^{2,p}(\R^d)\longrightarrow L^p(\R^{2d}),
	\end{equation*}
	satisfying \eqref{eq:lorentz-interpolation-conclusion}.
\end{lemma}

\begin{proof}
	For $f,g\in\mcS(\R^d)$ and $h\in\mcS(\R^{2d})$, define the trilinear family
	\begin{equation*}
		\Lambda_z(f,g,h)
		:=\int_{\R^{2d}}\mathcal T_z(f,\overline g)(X)h(X)\,\md{X}.
	\end{equation*}
	For fixed $f,g,h$ as above, pairing the $L^2$-valued map
	$\mathcal T_z(f,\overline g)$ with $h$ shows that
	$z\mapsto\Lambda_z(f,g,h)$ is holomorphic in the open strip and
	continuous on its closure. In particular, its boundary values are
	measurable. By \eqref{eq:lorentz-interpolation-growth} and the
	Cauchy--Schwarz inequality,
	\begin{equation*}
		|\Lambda_{\sigma+it}(f,g,h)|
		\leqslant C_{f,\overline g}\|h\|_2(1+|t|)^N,
		\qquad 0\leqslant\sigma\leqslant1,
	\end{equation*}
	which gives the required admissible growth. Let $q'=\frac{q}{q-1}$. The boundary estimates \eqref{eq:lorentz-left} and
	\eqref{eq:lorentz-right}, together with H\"older's inequality, give
     \begin{align*}
     		|\Lambda_{it}(f,g,h)|
     	&\leqslant A_0(1+|t|)^{N_0}
     	\|f\|_2\|g\|_2\|h\|_2,\\
     	|\Lambda_{1+it}(f,g,h)|
     	&\leqslant A_1(1+|t|)^{N_1}
     	\|f\|_{2,q}\|g\|_{2,q}\|h\|_{q'}.
     \end{align*}
     
	Consider the couples
	\begin{align*}
		X_{01}=X_{02}&=L^2(\R^d),
		&X_{11}=X_{12}&=L^{2,q}(\R^d),\\
		X_{03}&=L^2(\R^{2d}),
		&X_{13}&=L^{q'}(\R^{2d}),
	\end{align*}
	and the dense subspaces $\mathcal X_1=\mathcal X_2=\mcS(\R^d)$ and $\mathcal X_3=\mcS(\R^{2d})$.	For the target spaces, take $Y_0=Y_1=\C$ on the one-point measure
	space with mass one. 
	
	Let $p'=\frac{p}{p-1}$. Applying \cite[Theorem 1.1]{GM14} to this family at $z=\theta$, we obtain
			\begin{equation*}
			|\Lambda_\theta(f,g,h)|
			\leqslant C_{q,\theta,N_0,N_1}A_0^{1-\theta}A_1^\theta
			\|f\|_{2,p}\|g\|_{2,p}\|h\|_{p'}.
		\end{equation*}
		For fixed $f,g\in\mcS(\R^d)$, this shows that
		$h\mapsto\Lambda_\theta(f,g,h)$ extends to a bounded linear
		functional on $L^{p'}(\R^{2d})$. By duality, this functional
		is represented by a function $F_{f,g}\in L^p(\R^{2d})$ satisfying
		\begin{equation*}
			\|F_{f,g}\|_p
			\leqslant C_{q,\theta,N_0,N_1}A_0^{1-\theta}A_1^\theta
			\|f\|_{2,p}\|g\|_{2,p}.
		\end{equation*}
		By the definition of $\Lambda_\theta$, $F_{f,g}$ and
		$\mathcal T_\theta(f,\overline g)\in L^2(\R^{2d})$ define
		the same distribution and hence agree almost everywhere.
		Replacing $g$ by $\overline g$ gives
		\eqref{eq:lorentz-interpolation-conclusion}. Since $\mcS(\R^d)$ is dense in $L^{2,p}(\R^d)$, \eqref{eq:lorentz-interpolation-conclusion} and sesquilinearity give a continuous sesquilinear extension satisfying the same estimate.
		The uniqueness also follows from density.
\end{proof}

\subsection{Gaussian estimates and dyadic summation}

We first estimate the cross-$\tau$-Wigner distributions of the Gaussian functions defined below. See \cite[Lemma 1.3.34]{CR20} for a related computation for the cross-Wigner distribution
($\tau=\frac{1}{2}$).

For $r>0$, define
\begin{equation}\label{eq:gaussian}
	\phi_r(x):=r^{-\frac{d}{2}}e^{-\pi\frac{|x|^2}{r^2}}.
\end{equation}
\begin{lemma}\label{lem:gaussian-pair}
	Let $d\geqslant1$ and $0<\alpha<d$, and let
	$$
	P_r(x):=r^{\alpha-\frac{d}{2}}\left(1+\frac{|x|^2}{r^2}\right)^{-\frac{d-\alpha}{4}}.
	$$
	There is $C_{d,\alpha}>0$ such that, for every $r,s>0$ and $(x,\xi)\in\R^{2d}$,
	\begin{equation*}
		\int_0^{\frac{1}{2}}\tau^{\alpha-1}
		|W_\tau(\phi_r,\phi_s)(x,\xi)|\md{\tau}
		\leqslant C_{d,\alpha}P_r(x)P_{s^{-1}}(\xi).
	\end{equation*}
\end{lemma}

\begin{proof}
	By \eqref{eq:wtau} and \eqref{eq:gaussian},
	\begin{equation*}
		W_\tau(\phi_r,\phi_s)(x,\xi)
		=(rs)^{-\frac{d}{2}}\int_{\R^d}e^{-2\pi i\xi\cdot y}
		\exp\left(-\pi\frac{|x+\tau y|^2}{r^2}
		-\pi\frac{|x-(1-\tau)y|^2}{s^2}\right)\md{y}.
	\end{equation*}
	Set
	$$
	D=r^2(1-\tau)^2+s^2\tau^2,
	\qquad
	b=\frac{s^2\tau-r^2(1-\tau)}{D}.
	$$
	Since $r,s>0$, we have $D>0$. Expanding the two quadratic
	terms gives
	\begin{align*}
		\frac{|x+\tau y|^2}{r^2}+\frac{|x-(1-\tau)y|^2}{s^2}
		&=\left(\frac{1}{r^2}+\frac{1}{s^2}\right)|x|^2+2\left(
		\frac{\tau}{r^2}-\frac{1-\tau}{s^2}\right)x\cdot y
		\\
		&\qquad+\left(\frac{\tau^2}{r^2}+\frac{(1-\tau)^2}{s^2}\right)|y|^2\\
		&=\frac{r^2+s^2}{r^2s^2}|x|^2+\frac{2bD}{r^2s^2}x\cdot y+\frac{D}{r^2s^2}|y|^2\\
		&=\frac{D}{r^2s^2}|y+bx|^2+\left(\frac{r^2+s^2}{r^2s^2}-\frac{Db^2}{r^2s^2}\right)|x|^2.
	\end{align*}
	Since $D^2b^2=\bigl[s^2\tau-r^2(1-\tau)\bigr]^2$, a direct computation gives
	\begin{align*}
		\frac{r^2+s^2}{r^2s^2}-\frac{Db^2}{r^2s^2}=\frac{(r^2+s^2)D-D^2b^2}{Dr^2s^2}=\frac{(r^2+s^2)D-\bigl[s^2\tau-r^2(1-\tau)\bigr]^2}{Dr^2s^2}=\frac{1}{D}.
	\end{align*}
	Therefore we obtain
	$$
	W_\tau(\phi_r,\phi_s)(x,\xi)
	=(rs)^{-\frac{d}{2}}\int_{\R^d}e^{-2\pi i\xi\cdot y}
	\exp\left(-\pi \left[\frac{D}{r^2s^2}|y+bx|^2+\frac{|x|^2}{D}\right]\right)\md{y}.
	$$
	Substitute $u=y+bx$ and apply \eqref{eq:gaussian-fourier} with
	$t=\frac{D}{r^2s^2}$. We obtain
	\begin{equation*}
		W_\tau(\phi_r,\phi_s)(x,\xi)
		=(rs)^{\frac{d}{2}}D^{-\frac{d}{2}}
		\exp\left(-\pi\frac{|x|^2+r^2s^2|\xi|^2}{D}\right)
		e^{2\pi ibx\cdot\xi}.
	\end{equation*}
	Since $e^{-\pi u}\leqslant C_d(1+u)^{-\frac{d}{2}}$ for $u\geqslant0$,
	\begin{equation*}
		|W_\tau(\phi_r,\phi_s)(x,\xi)|
		\leqslant C_d(rs)^{\frac{d}{2}}
		\bigl(D+|x|^2+r^2s^2|\xi|^2\bigr)^{-\frac{d}{2}}.
	\end{equation*}
	For $0<\tau<\frac{1}{2}$, 
	$$
	D=r^2(1-\tau)^2+s^2\tau^2\geqslant \frac{r^2}{4}+s^2\tau^2.
	$$ 
    Set $\Lambda:=r^2+|x|^2+r^2s^2|\xi|^2$. Then $D+|x|^2+r^2s^2|\xi|^2>\frac{1}{4}(\Lambda +s^2\tau^2)$ and hence
	\begin{align*}
		\int_0^{\frac{1}{2}}\tau^{\alpha-1}|W_\tau(\phi_r,\phi_s)(x,\xi)|\md{\tau}&\leqslant C_d(rs)^{\frac{d}{2}}\int_0^\infty\tau^{\alpha-1}
		(\Lambda+s^2\tau^2)^{-\frac{d}{2}}\md{\tau}\\
		&=C_d(rs)^{\frac{d}{2}}s^{-\alpha}\Lambda^{\frac{\alpha-d}{2}}\times
		\int_0^\infty u^{\alpha-1}(1+u^2)^{-\frac{d}{2}}\md{u}\\
		&\leqslant C_{d,\alpha}(rs)^{\frac{d}{2}}s^{-\alpha}\Lambda^{\frac{\alpha-d}{2}}.
	\end{align*}
	Since $\Lambda=r^2+|x|^2+r^2s^2|\xi|^2=r^2(1+\frac{|x|^2}{r^2}+s^2|\xi|^2)$, 
	we obtain
	\begin{align*}
		(rs)^{\frac{d}{2}}s^{-\alpha}\Lambda^{\frac{\alpha-d}{2}}&=r^{\alpha-\frac{d}{2}}s^{\frac{d}{2}-\alpha}
		\left(1+\frac{|x|^2}{r^2}+s^2|\xi|^2\right)^{-\frac{d-\alpha}{2}}\\
		&\leqslant  r^{\alpha-\frac{d}{2}}s^{\frac{d}{2}-\alpha}
		(1+r^{-2}|x|^2)^{-\frac{d-\alpha}{4}}
		(1+s^2|\xi|^2)^{-\frac{d-\alpha}{4}}\\
		&=P_r(x)P_{s^{-1}}(\xi).
	\end{align*}
	Here we used that $(1+a)(1+b)\leqslant (1+a+b)^2$ for $a,b\geqslant0$.
\end{proof}

In order to use this estimate, we establish the following lemma, which states that for a nonnegative, bounded, compactly supported and symmetric decreasing function $h$, we can construct a majorant formed from a finite positive sum of Gaussians.

\begin{lemma}\label{lem:majorant}
	Let $d\geqslant1$ and $1\leqslant q<\infty$. Suppose that $h(x)=H(|x|)$ is
	nonnegative, bounded and compactly supported, and that $H$ is
	nonincreasing. Then there exists a finitely supported sequence of
	nonnegative numbers $(a_k)_{k\in\mathbb Z}$ such that 
\begin{equation*}
h\leqslant\sum_k a_k\phi_{2^k}\quad\text{almost everywhere},\qquad	 \|(a_k)\|_{\ell^q}\leqslant C_{d,q}\|h\|_{2,q},
\end{equation*}	
where  $\phi_r$ is given by \eqref{eq:gaussian} and $C_{d,q}$ depends only on $d$ and $q$.
\end{lemma}
\begin{proof}
	For $h=0$ take all coefficients to be zero. Otherwise, let
	$M:=\|h\|_\infty$ and write
	\begin{equation*}
		d_h(\lambda):=|\{x\in\R^d:h(x)>\lambda\}|,
		\qquad \lambda>0.
	\end{equation*}
	\begin{equation*}
		\|h\|_{2,q}^q
		=2\int_0^\infty\lambda^{q-1}d_h(\lambda)^{\frac{q}{2}}\md{\lambda}.
	\end{equation*}

    Choose a sufficiently small $k_0\in\Z$ such that $e^\pi2^{\frac{k_0d}{2}}M\leqslant\|h\|_{2,q}$ and choose $K\in\Z$, $K>k_0$, such that $h=0$ almost everywhere outside
	$B(0,2^K)$. Define
	\begin{equation*}
		a_{k_0}:=e^\pi2^{\frac{k_0d}{2}}M,
		\qquad
		a_k:=e^\pi2^{\frac{kd}{2}}H(2^{k-1})
		\quad(k_0+1\leqslant k\leqslant K),
	\end{equation*}
	and set $a_k=0$ for all remaining indices.
	
	For $0\leqslant |x| < 2^{k_0}$,  we have
	\begin{equation*}
		a_{k_0}\phi_{2^{k_0}}(x)
	=e^\pi M e^{-\pi\frac{|x|^2}{2^{2k_0}}}
		\geqslant M\geqslant h(x)
		\quad\text{almost everywhere}.
	\end{equation*}
	For $2^{k-1}\leqslant|x|<2^k$, where $k_0+1\leqslant k\leqslant K$, monotonicity of
	$H$ gives
	\begin{equation*}
		a_k\phi_{2^k}(x)
		=e^\pi H(2^{k-1})e^{-\pi\frac{|x|^2}{2^{2k}}}
		\geqslant H(2^{k-1})\geqslant H(|x|).
	\end{equation*}
	Thus $h\leqslant\sum_k a_k\phi_{2^k}$ almost everywhere.
	
	It remains to estimate $\|(a_k)\|_{\ell^q}$. For $\lambda>0$, let
	\begin{equation*}
		I_\lambda:=\{k\in\Z:k_0+1\leqslant k\leqslant K,\ H(2^{k-1})>\lambda\}.
	\end{equation*}
	If $k\in I_\lambda$, then $B(0,2^{k-1})\subset\{h>\lambda\}$ up to a
	null set, and hence $2^{(k-1)d}\omega_d\leqslant d_h(\lambda)$.
	If $I_\lambda$ is nonempty, let $m:=\max I_\lambda$. Summing a geometric
	series, we obtain
	\begin{align*}
		\sum_{k\in I_\lambda}2^{k\frac{dq}{2}}&\leqslant\sum_{k=-\infty}^m2^{k\frac{dq}{2}}
		=\frac{2^{m\frac{dq}{2}}}{1-2^{-\frac{dq}{2}}}\leqslant C_{d,q} d_h(\lambda)^{\frac{q}{2}}.
	\end{align*}
	The same bound holds when $I_\lambda$ is empty. By Tonelli's theorem,
	\begin{align*}
		\sum_{k=k_0+1}^K a_k^q&=e^{\pi q}\sum_{k=k_0+1}^K
		2^{k\frac{dq}{2}}H(2^{k-1})^q=e^{\pi q}q\int_{0}^\infty \lambda^{q-1}\sum_{k=k_0+1}^K2^{k\frac{dq}{2}}1_{\{H(2^{k-1})>\lambda\}}d\lambda\\
		&=C_{d,q}\int_0^\infty\lambda^{q-1}
		\sum_{k\in I_\lambda}2^{k\frac{dq}{2}}\md{\lambda}\leqslant
		C_{d,q}\int_0^\infty\lambda^{q-1}d_h(\lambda)^{\frac{q}{2}}\md{\lambda}=C_{d,q}
		\|h\|_{2,q}^q.
	\end{align*}
	Since $a_{k_0}^q\leqslant\|h\|_{2,q}^q$, we conclude that
	\begin{equation*}
		\|(a_k)\|_{\ell^q}^q=a_{k_0}^q+\sum_{k=k_0+1}^K a_k^q\leqslant C_{d,q}\|h\|_{2,q}^q.
	\end{equation*}
	Taking the $q$-th root proves the result.
\end{proof}

The following lemma is a dyadic summation estimate for the functions
$P_r$ introduced in Lemma \ref{lem:gaussian-pair}, which will be needed later.

\begin{lemma}\label{lem:profiles}
	Let $d\geqslant1$ and $2<q<\infty$, and set
	\begin{equation*}
		\alpha=d\left(\frac{1}{2}-\frac{1}{q}\right).
	\end{equation*}
	Let $P_r$ be as in Lemma \ref{lem:gaussian-pair}
	with this value of $\alpha$.
	There is a constant $C_{d,q}>0$ such that, for every
	finitely supported sequence $(b_k)_{k\in\Z}$,
	\begin{equation}\label{eq:profile-synthesis}
		\left\|\sum_k b_kP_{2^k}\right\|_{L^q(\R^d)}
		\leqslant C_{d,q}\|(b_k)\|_{\ell^q}.
	\end{equation}
\end{lemma}

\begin{proof}
Fix $x\ne0$ and choose $n\in\Z$ such that
	$2^n\leqslant|x|<2^{n+1}$.
	By the definition of $P_r$,
	\begin{equation*}
		P_{2^k}(x)=2^{k\alpha-\frac{dk}{2}}\left(1+\frac{|x|^2}{2^{2k}}\right)^{-\frac{d-\alpha}{4}}\leqslant
		\begin{cases}
			2^{\frac{k\alpha}{2}}|x|^{-\frac{d-\alpha}{2}},
			& k\leqslant n,\\
			2^{k(\alpha-\frac{d}{2})},
			& k\geqslant n+1.
		\end{cases}
	\end{equation*}
	Summing the two geometric series gives
	\begin{align*}
		\sum_{k\in\Z}P_{2^k}(x)\leqslant
		|x|^{-\frac{d-\alpha}{2}}
		\sum_{k\leqslant n}2^{\frac{k\alpha}{2}}
		+\sum_{k\geqslant n+1}2^{k(\alpha-\frac{d}{2})}\leqslant C_{d,q}|x|^{\alpha-\frac{d}{2}}.
	\end{align*}
	By H\"older's inequality,
	\begin{align*}
		\left|\sum_k b_kP_{2^k}(x)\right|^q\leqslant
		\left(\sum_k|b_k|^qP_{2^k}(x)\right)
		\left(\sum_{k\in\Z}P_{2^k}(x)\right)^{q-1}\leqslant
		C_{d,q}\sum_k|b_k|^qP_{2^k}(x)
		|x|^{-\frac{d}{2}-\alpha},
	\end{align*}
	where we used $(q-1)\left(\alpha-\frac{d}{2}\right)
	=-\frac{d}{2}-\alpha$.
	
	Integrating in $x$, applying Tonelli's theorem,
	and using polar coordinates, we obtain
	\begin{align*}
		\left\|\sum_k b_kP_{2^k}\right\|_{L^q(\R^d)}^q
		&\leqslant
		C_{d,q}\sum_k|b_k|^q
		\int_{\R^d}P_{2^k}(x)|x|^{-\frac{d}{2}-\alpha}\md{x}\\
		&=C_{d,q}|\mathbb S^{d-1}|
		\sum_k|b_k|^q2^{k(\alpha-\frac{d}{2})}
		\int_0^\infty
		\left(1+\frac{\rho^2}{2^{2k}}\right)^{-\frac{d-\alpha}{4}}
		\rho^{\frac{d}{2}-\alpha-1}\md{\rho}\\
		&=C_{d,q}|\mathbb S^{d-1}|
		\sum_k|b_k|^q
		\int_0^\infty
		(1+u^2)^{-\frac{d-\alpha}{4}}
		u^{\frac{d}{2}-\alpha-1}\md{u}.
	\end{align*}
		Since $2<q<\infty$, we have $0<\alpha<\frac{d}{2}$ and hence
	\begin{align*}
		\int_0^\infty
		(1+u^2)^{-\frac{d-\alpha}{4}}
		u^{\frac{d}{2}-\alpha-1}\md{u}\leqslant
		\int_0^1u^{\frac{d}{2}-\alpha-1}\md{u}
		+\int_1^\infty u^{-1-\frac{\alpha}{2}}\md{u}<\infty.
	\end{align*}
	Consequently,
	\begin{equation*}
		\left\|\sum_k b_kP_{2^k}\right\|_{L^q(\R^d)}^q
		\leqslant C_{d,q}\sum_k|b_k|^q.
	\end{equation*}
	Taking the $q$-th root proves \eqref{eq:profile-synthesis}.
\end{proof}

\subsection{Sharpness of the Lorentz indices}
Recall that $d\geqslant3$ and $p^*=\frac{2d}{d-2}$.

\begin{proposition}\label{prop:lorentz-sharpness}
	For $0<r_1,r_2\leqslant\infty$, suppose that
	\begin{equation}\label{eq:lorentz-general}
		\|W_{BJ}(f,g)\|_{p^*}
		\leqslant C\|f\|_{2,r_1}\|g\|_{2,r_2}
	\end{equation}
	holds for all bounded, compactly supported functions $f$ and $g$. Then 
	\begin{equation*}
		r_1\leqslant p^*,\qquad r_2\leqslant p^*.
	\end{equation*}
	Thus the Lorentz indices in the endpoint estimate are sharp.
\end{proposition}

\begin{proof}
	By the symmetry relation \eqref{eq:reflection}, we have
	\begin{equation*}
		W_{BJ}(f,g)=\overline{W_{BJ}(g,f)}.
	\end{equation*}
	Therefore it suffices to prove that $r_1\leqslant p^*$. For an integer $N\geqslant3$, define
	\begin{equation*}
		f_N(x):=\bigl(\max\{|x|,2^{-N}\}\bigr)^{-\frac d2}
		1_{B(0,1)},\qquad g(x):=1_{B(0,1)}.
	\end{equation*}
	Both functions are bounded and compactly supported. A direct computation gives
	\begin{equation*}
		d_{f_N}(\lambda)=|\{x: f_N(x)> \lambda\}|=
		\begin{cases}
			\omega_d,&0<\lambda<1,\\[1mm]
			\omega_d\lambda^{-2},&1\leqslant\lambda<2^{\frac{Nd}{2}},\\[1mm]
			0,&\lambda\geqslant2^{\frac{Nd}{2}}.
		\end{cases}
	\end{equation*}
	For $0<r_1<\infty$, the definition of the Lorentz quasi-norm gives
	\begin{align*}
		\|f_N\|_{2,r_1}^{r_1}=2\int_0^\infty \lambda^{r_1-1}d_{f_N}(\lambda)^{\frac{r_1}{2}}\md \lambda=2\omega_d^{\frac{r_1}{2}}\int_0^1
		\lambda^{r_1-1}\md{\lambda}
		+2\omega_d^{\frac{r_1}{2}}
		\int_1^{2^{\frac{Nd}{2}}}\frac{\md{\lambda}}{\lambda}=\omega_d^{\frac{r_1}{2}}
		\left(\frac{2}{r_1}+Nd\log2\right).
	\end{align*}
	Consequently, $\|f_N\|_{2,r_1}\leqslant C_{d,r_1}N^{\frac1{r_1}}$.
	If $r_1=\infty$, then 
	\begin{equation*}
		\|f_N\|_{2,\infty}
		=\sup_{\lambda>0}\lambda d_{f_N}(\lambda)^{\frac12}
		=\omega_d^{\frac12}.
	\end{equation*}
	Moreover, $\|g\|_{2,r_2}$ is finite and independent of $N$.
	
	We next obtain a lower bound for the Born-Jordan distribution. Since
	$f_N$ and $g$ are nonnegative and real-valued,
	\begin{align*}
		\operatorname{Re}W_{BJ}(f_N,g)(x,\xi)=\int_0^1\int_{\R^d}
		\cos(2\pi\xi\cdot y)
		f_N(x+\tau y)g(x-(1-\tau)y)\md{y}\md{\tau}.
	\end{align*}
	Whenever the integrand is nonzero, we have $x+\tau y, x-(1-\tau)y \in B(0,1)$, and therefore
	\begin{equation*}
		|y|
		=|(x+\tau y)-(x-(1-\tau)y)|
		\leqslant2.
	\end{equation*}
	Hence, if $|\xi|\leqslant\frac1{12}$, then $|2\pi\xi\cdot y|\leqslant\frac\pi3$ and  $\cos(2\pi\xi\cdot y)\geqslant\frac12$. Now we set
	\begin{equation*}
	 \Omega_1=\left\{(x,\xi)\in \R^d\times \R^d: 2^{-N}\leqslant|x|\leqslant\frac14,
	 	\,|\xi|\leqslant\frac1{12}\right\},
	\end{equation*}
	and, for fixed $(x,\xi)\in \Omega_1$,
	\begin{equation*}
    \Omega_2=\left\{(\tau, y)\in \R\times \R^d: 0\leqslant\tau\leqslant |x|,\, |y|\leqslant\frac{1}{2}\right\}.
	\end{equation*}
	If $(\tau, y)\in \Omega_2$, we have $|x+\tau y|\leqslant \frac{3}{2}|x|$ and $|x-(1-\tau)y|\leqslant \frac{3}{4}$. Therefore for $(x,\xi)\in \Omega_1$,
	\begin{align*}
		\operatorname{Re}W_{BJ}(f_N,g)(x,\xi)&\geqslant\int_0^{|x|}\int_{|y|\leqslant \frac{1}{2}}\cos(2\pi\xi\cdot y)f_N(x+\tau y)g(x-(1-\tau)y)\md{y}\md{\tau}\\
		&\geqslant \frac{1}{2}\int_0^{|x|}\int_{|y|\leqslant\frac{1}{2}}\left(\frac{3}{2}|x|\right)^{-\frac{d}{2}}\md y\md \tau \geqslant c_d |x|^{1-\frac{d}{2}}.
	\end{align*}
	Therefore $|W_{BJ}(f_N,g)(x,\xi)|\geqslant c_d|x|^{1-\frac{d}{2}}$ on $\Omega_1$. Using 
 that $(1-\frac d2)p^*=-d$, we obtain
	\begin{align*}
		\|W_{BJ}(f_N,g)\|_{p^*}^{p^*}
		&\geqslant c_d
		\int_{|\xi|\leqslant\frac1{12}}
		\int_{2^{-N}\leqslant|x|\leqslant\frac14}
		|x|^{-d}\md{x}\md{\xi}\\
		&\geqslant c_d\int_{2^{-N}}^{\frac14}\frac{\md{\rho}}{\rho}
		=c_d(N-2)\log2,
	\end{align*}
	which implies $\|W_{BJ}(f_N,g)\|_{p^*}\geqslant c_dN^{\frac1{p^*}}$.
	
	Assume first that $r_1<\infty$. Applying \eqref{eq:lorentz-general}
	to $(f_N,g)$, we obtain
	\begin{equation*}
		N^{\frac1{p^*}}
		\leqslant C_{d,r_1,r_2}N^{\frac1{r_1}}.
	\end{equation*}
	Letting $N\to\infty$ gives
	\begin{equation*}
		\frac1{p^*}\leqslant\frac1{r_1},
	\end{equation*}
	and therefore $r_1\leqslant p^*$. If $r_1=\infty$, then we obtain $N^{1/p^*}\leqslant C_{d,r_2}$, which is impossible as $N\to\infty$. Thus $r_1\leqslant p^*$ in all cases.
\end{proof}

\section{Weighted even-moment estimates}\label{sec:even-moment}

For $v\in L^1(0,\frac{1}{2})$ and $f,g\in L^2(\R^d)$, set
\begin{equation*}
	Q_v(f,g):=\int_0^{\frac{1}{2}}v(\tau)W_\tau(f,g)\md{\tau}.
\end{equation*}
By Lemma \ref{lem:moyal}, this integral is well-defined as an
$L^2(\R^{2d})$-valued integral.

\begin{theorem}[Weighted even-moment estimates]\label{thm:weighted}
	Let $d\geqslant3$ and let $q\geqslant4$ be an even integer. Set
	$\alpha=d(\frac{1}{2}-\frac{1}{q})$. Suppose $v:(0,\frac{1}{2})\to\C$ is measurable and
	\begin{equation*}
		|v(\tau)|\leqslant M\tau^{\alpha-1},
		\quad\text{almost everywhere for some }M\geqslant0.
	\end{equation*}
	Then $Q_v$ has a unique bounded sesquilinear extension
	\begin{equation*}
		Q_v:L^{2,q}(\R^d)\times L^{2,q}(\R^d)\longrightarrow L^q(\R^{2d}).
	\end{equation*}
	For every $f,g\in L^{2,q}(\R^d)$, this extension satisfies
	\begin{equation*}
		\|Q_v(f,g)\|_q\leqslant C_{d,q}M\|f\|_{2,q}\|g\|_{2,q}.
	\end{equation*}
\end{theorem}

With $d,q,\alpha$ as in Theorem \ref{thm:weighted}, set
$w(\tau):=\tau^{\alpha-1}$. Since $\alpha>0$, we have
$w\in L^1(0,\frac{1}{2})$, and hence $v\in L^1(0,\frac{1}{2})$.

\begin{proof}[Proof of Theorem \ref{thm:weighted}]
	The assertion is immediate if $M=0$, $f=0$, or $g=0$.
	By replacing $v$ with $\frac{v}{M}$, it suffices to take $M=1$.
	
	We first consider bounded, compactly supported $f$ and $g$. Thus $f,g\in L^2$ and $Q_v(f,g)$ is well-defined. Since rearrangement preserves the Lorentz quasi-norms, Lemma \ref{lem:majorant} gives finite nonnegative Gaussian sums such that
	\begin{equation}\label{eq:input-majorants}
		f^*\leqslant F=\sum_k a_k\phi_{2^k},
		\qquad
		g^*\leqslant G=\sum_l b_l\phi_{2^l},
	\end{equation}
	where $(a_k)_{k\in \Z}$ and $(b_l)_{l\in\Z}$ are finitely supported nonnegative sequences and satisfy
	\begin{equation}\label{eq: a_k b_l}
		\|(a_k)\|_{\ell^q}\leqslant C_{d,q}\|f\|_{2,q},\qquad
		\|(b_l)\|_{\ell^q}\leqslant C_{d,q}\|g\|_{2,q}.
	\end{equation}
	
	Set $\sigma_j=1$ for $1\leqslant j\leqslant \frac{q}{2}$ and
	$\sigma_j=-1$ for $\frac{q}{2}+1\leqslant j\leqslant q$. Write
	$\md{\mathbf y}=\md{y}_1\cdots\md{y}_q$ and
	$\md{\boldsymbol\tau}=\md{\tau}_1\cdots\md{\tau}_q$. Fix $\delta>0$.
	Choose $R>0$ such that $f=g=0$ almost everywhere outside $B(0,R)$.
	If the product in the following integral is nonzero, then
	$|x|\leqslant R$ and $|y_j|\leqslant2R$ for $1\leqslant j\leqslant q$.
	Since $f,g$ are bounded and $v\in L^1(0,\frac{1}{2})$, the Gaussian
	factor in $\xi$ makes the integral absolutely convergent.
	Expanding $Q_v^{\frac{q}{2}}\overline{Q_v}^{\,\frac{q}{2}}$ and using
	Fubini's theorem, we have
    \begin{align*}
    		\iint_{\R^{2d}}& e^{-\pi\delta|\xi|^2}
    	|Q_v(f,g)(x,\xi)|^q\md{x}\md{\xi}=\iint_{\R^{2d}} e^{-\pi\delta|\xi|^2}
    	Q_v^{\frac{q}{2}}(f,g)\overline{Q_v(f,g)}^{\,\frac{q}{2}}\md{x}\md{\xi}\\
    	&=\int_{(0,\frac{1}{2})^q}\int_{\R^{(q+1)d}}\int_{\R^d}e^{-\pi\delta|\xi|^2}
    	e^{-2\pi i\xi\cdot\sum_{j=1}^q\sigma_jy_j}\\
    	&\quad\times\prod_{j=1}^{\frac{q}{2}}
    	\left[v(\tau_j)f(x+\tau_jy_j)
    	\overline{g(x-(1-\tau_j)y_j)}\right]\\
    	&\quad\times\prod_{j=\frac{q}{2}+1}^{q}
    	\left[\overline{v(\tau_j)}\,\overline{f(x+\tau_jy_j)}
    	g(x-(1-\tau_j)y_j)\right]
    	\md{\xi}\md{x}\md{\mathbf y}\md{\boldsymbol\tau}.
    \end{align*}
	Integrating in $\xi$ and using \eqref{eq:gaussian-fourier} give
	\begin{align*}
		\int_{\R^d}e^{-\pi\delta|\xi|^2}
		e^{-2\pi i\xi\cdot\sum_{j=1}^q\sigma_jy_j}\md{\xi}= H_\delta\Bigl(\sum_{j=1}^q \sigma_j y_j\Bigr),\qquad  H_\delta(y):=\delta^{-\frac{d}{2}}e^{-\pi\frac{|y|^2}{\delta}}.
	\end{align*}
	After integrating in $\xi$, we take absolute values and use
	$|v|\leqslant w$ to obtain
	\begin{align*}
		\iint_{\R^{2d}}&e^{-\pi\delta|\xi|^2}
		|Q_v(f,g)(x,\xi)|^q\md{x}\md{\xi}\leqslant\int_{(0,\frac{1}{2})^q}\prod_{j=1}^q w(\tau_j)
		\int_{\R^{(q+1)d}}
		H_\delta\left(\sum_{j=1}^q\sigma_jy_j\right)\\
		&\hspace{25mm}\times
		\prod_{j=1}^q
		|f(x+\tau_jy_j)|\,|g(x-(1-\tau_j)y_j)|
		\md{x}\md{\mathbf y}\md{\boldsymbol\tau}.
	\end{align*}
	Using the Brascamp-Lieb-Luttinger inequality (Lemma \ref{lem:bll}), we obtain
	\begin{align*}
		\iint_{\R^{2d}}&e^{-\pi\delta|\xi|^2}
		|Q_v(f,g)(x,\xi)|^q\md{x}\md{\xi}\leqslant\int_{(0,\frac{1}{2})^q}\prod_{j=1}^q w(\tau_j)
		\int_{\R^{(q+1)d}}
		H_\delta\left(\sum_{j=1}^q\sigma_jy_j\right)\\
		&\hspace{25mm}\times
		\prod_{j=1}^q
		f^*(x+\tau_jy_j)\,g^*(x-(1-\tau_j)y_j)
		\md{x}\md{\mathbf y}\md{\boldsymbol\tau}\\
		&=\iint_{\R^{2d}}e^{-\pi\delta|\xi|^2}
		|Q_w(f^*,g^*)(x,\xi)|^q\md{x}\md{\xi}\leqslant\iint_{\R^{2d}}e^{-\pi\delta|\xi|^2}
		|Q_w(F,G)(x,\xi)|^q\md{x}\md\xi.
	\end{align*}
	For the last inequality, we replace $f^*,g^*$ by $F,G$
	in the preceding nonnegative-kernel integral and then
	reverse the expansion. Thus we have proved
	\begin{equation}\label{eq:moment-comparison}
		\iint_{\R^{2d}}e^{-\pi\delta|\xi|^2}
		|Q_v(f,g)(x,\xi)|^q\md{x}\md{\xi}\leqslant
		\iint_{\R^{2d}}e^{-\pi\delta|\xi|^2}
		|Q_w(F,G)(x,\xi)|^q\md{x}\md{\xi}.
	\end{equation}
	By \eqref{eq:input-majorants} and Lemma \ref{lem:gaussian-pair}, we have
	\begin{align*}
		|Q_w(F,G)(x,\xi)|&\leqslant\sum_{k,l}a_kb_l
		\int_0^{\frac{1}{2}}\tau^{\alpha-1}
		|W_\tau(\phi_{2^k},\phi_{2^l})(x,\xi)|\md{\tau}\\
		&\leqslant C_{d,q}
		\Bigl(\sum_k a_kP_{2^k}(x)\Bigr)
		\Bigl(\sum_l b_lP_{2^{-l}}(\xi)\Bigr).
	\end{align*}
	Applying Tonelli's theorem and Lemma \ref{lem:profiles}, and using \eqref{eq: a_k b_l}, we obtain
	\begin{equation*}
		\begin{split}
			\|Q_w(F,G)\|_q
			&\leqslant C_{d,q}
			\left\|\sum_k a_kP_{2^k}\right\|_{L^q}
			\left\|\sum_l b_lP_{2^{-l}}\right\|_{L^q}\\
			&\leqslant C_{d,q}\|(a_k)\|_{\ell^q}\|(b_l)\|_{\ell^q}\leqslant  C_{d,q}\|f\|_{2,q}\|g\|_{2,q}.
		\end{split}
	\end{equation*}
	This bound is independent of $\delta$. Combining this bound with \eqref{eq:moment-comparison}
	and letting $\delta\downarrow0$, we obtain, by the
	monotone convergence theorem,
	$$
	\left(\iint_{\R^{2d}}|Q_v(f,g)(x,\xi)|^q\md{x}\md{\xi}\right)^{\frac{1}{q}}
	\leqslant C_{d,q}\|f\|_{2,q}\|g\|_{2,q}
	$$
	for bounded and compactly supported $f$ and $g$.
	
	For general $f,g\in L^{2,q}(\R^d)$, set
	\begin{equation*}
		f_n=f\,1_{\{|x|<n\}}1_{\{|f|<n\}},\qquad
		g_n=g\,1_{\{|x|<n\}}1_{\{|g|<n\}}.
	\end{equation*}
	Since $q<\infty$, we have $f_n\to f$ and $g_n\to g$
	in $L^{2,q}(\R^d)$. Moreover, we have
	\begin{equation*}
		\|f_n\|_{2,q}\leqslant\|f\|_{2,q},\qquad
		\|g_n\|_{2,q}\leqslant\|g\|_{2,q}.
	\end{equation*}
	
	By sesquilinearity and the estimate for bounded, compactly supported
	functions,
	\begin{align*}
		\|Q_v(f_n,g_n)-Q_v(f_m,g_m)\|_q\leqslant C_{d,q}
		\bigl(\|f_n-f_m\|_{2,q}\|g_n\|_{2,q}
		+\|f_m\|_{2,q}\|g_n-g_m\|_{2,q}\bigr)
		\longrightarrow0.
	\end{align*}
	Since $L^q(\R^{2d})$ is complete, we may define
	\begin{equation*}
		\widetilde Q_v(f,g):=\lim_{n\to\infty}Q_v(f_n,g_n)
		\quad\text{in }L^q(\R^{2d}).
	\end{equation*}
	The same difference estimate shows that the limit is independent
	of the approximating sequences; sesquilinearity then follows
	by passage to the limit. Taking the limit in the preceding bound for $\|Q_v(f_n,g_n)\|_q$, we obtain
	\begin{equation*}
		\|\widetilde Q_v(f,g)\|_q
		\leqslant C_{d,q}\|f\|_{2,q}\|g\|_{2,q}.
	\end{equation*}
	
	It remains to check that this extension agrees with the original
	$L^2$-valued integral. Let now $f,g\in L^2(\R^d)\subset L^{2,q}(\R^d)$.
	The same truncations also converge in $L^2$. By sesquilinearity of the original map $Q_v$ and Lemma \ref{lem:moyal},
	\begin{align*}
		\|Q_v(f_n,g_n)-Q_v(f,g)\|_2
		&\leqslant\|v\|_1
		\bigl(\|f_n-f\|_2\|g_n\|_2
		+\|f\|_2\|g_n-g\|_2\bigr)
		\longrightarrow0.
	\end{align*}
	For every measurable set $B\subset\R^{2d}$ of finite measure,
	H\"older's inequality gives
	\begin{align*}
		\|Q_v(f,g)-\widetilde Q_v(f,g)\|_{L^1(B)}
		&\leqslant |B|^{\frac{1}{2}}
		\|Q_v(f,g)-Q_v(f_n,g_n)\|_2\\
		&\quad+|B|^{1-\frac{1}{q}}
		\|Q_v(f_n,g_n)-\widetilde Q_v(f,g)\|_q
		\longrightarrow0.
	\end{align*}
	Hence $\widetilde Q_v(f,g)=Q_v(f,g)$ almost everywhere,
	and we denote the extension still by $Q_v$.
	Uniqueness follows from the density of bounded, compactly
	supported functions in $L^{2,q}$.
\end{proof}

\section{Endpoint estimate}\label{sec:proof}

Throughout this section let $d\geqslant3$, $a=\frac{3d}{8}$, $\theta=\frac{1}{a}=\frac{8}{3d}<1$. For $f,g\in L^2(\R^d)$ and $\mathrm{Re} z>0$, define the $L^2$-valued integral
\begin{equation}\label{eq:half-family}
 T_z(f,g)=az\int_0^{\frac{1}{2}}\tau^{az-1}W_\tau(f,g)\md{\tau}.
\end{equation}
This is well defined by Lemma \ref{lem:moyal} and we have
\begin{equation*}
 T_\theta(f,g)=\int_0^{\frac{1}{2}}W_\tau(f,g)\md{\tau}.
\end{equation*}
It follows from the definition of $W_{BJ}$ and the symmetry relation \eqref{eq:reflection} that
\begin{equation}\label{eq:half-reconstruction}
 W_{BJ}(f,g)=T_\theta(f,g)+\overline{T_\theta(g,f)}.
\end{equation}

\subsection{Analytic family of operators and the left boundary estimate}
\begin{proposition}\label{prop:imaginary}
For fixed $f,g\in L^2(\R^d)$, the function
$z\mapsto T_z(f,g)$ is strongly holomorphic in $L^2(\R^{2d})$
on $\mathrm{Re} z>0$. It has an $L^2$-continuous extension to
$\mathrm{Re} z\geqslant0$, with $T_0(f,g)=W_0(f,g)$, and
\begin{equation}\label{eq:half-L2-bound}
 \|T_z(f,g)\|_2
 \leqslant2(1+|az|)^2\|f\|_2\|g\|_2,
 \qquad \mathrm{Re} z\geqslant0.
\end{equation}
\end{proposition}

	\begin{proof}
		Set $U(t)=W_{e^{-t}/2}(f,g), t\geqslant 0$. By the proof of Lemma \ref{lem:moyal}, $U$ is strongly continuous.
		
		For $\mathrm{Re} z>0$, making the change of variables
		$\tau=\frac{e^{-t}}{2}$ in \eqref{eq:half-family} gives
		$$
		T_z(f,g)
		=
		az\,2^{-az}
		\int_0^\infty e^{-azt}U(t)\md{t}.
		$$
		Thus $T_z(f,g)$ is, up to the entire factor $az\,2^{-az}$,
		the $L^2(\R^{2d})$-valued Laplace transform of $U$.
		
		Moreover, by \eqref{eq:moyal},
		$$
		\|U(t)\|_{L^2(\R^{2d})}=\|W_{e^{-t}/2}(f,g)\|_2=\|f\|_2\|g\|_2,
		$$
	    and hence, for every $\beta>0$,
	    $$
	    \int_0^\infty e^{-\beta t}\|U(t)\|_{L^2(\R^{2d})}\md{t}=\beta^{-1}\|f\|_2\|g\|_2<\infty.
	    $$
		Hence the vector-valued Laplace-transform theorem
		\cite[Theorem 1.5.1]{ABHN11} implies that $z\mapsto T_z(f,g)$ is
	$L^2(\R^{2d})$-valued holomorphic on $\mathrm{Re} z>0$.
	
    We next show that it can be continuously extended to the left boundary $\mathrm{Re} z=0$. For $\mathrm{Re} z\geqslant 0$ and $s\in\R$, set
   $$
   m_z(s)=2^{-az}+az\int_0^{\frac{1}{2}}\tau^{az-1}\left(e^{2\pi i s\tau}-1\right)\md{\tau}.
   $$
The integral converges at zero since $\left|e^{2\pi i s\tau}-1\right|	\leqslant2\pi |s|\tau$. Moreover, for fixed $s$, 
$$
\left|\int_0^{\frac{1}{2}}(\tau^{az_1-1}-\tau^{az_2-1})\left(e^{2\pi i s\tau}-1\right)\md{\tau}\right|\leqslant 2\pi |s|\int_0^{\frac{1}{2}}|\tau^{az_1}-\tau^{az_2}|\md{\tau},
$$
which implies that $z\mapsto m_z(s)$ is continuous on
$\mathrm{Re}z\geqslant0$.

We claim that
$$
|m_z(s)|\leqslant2(1+|az|)^2,\qquad\mathrm{Re} z\geqslant 0,\quad s\in\R.
$$
When $\pi |s|\leqslant 1$, using $\left|e^{2\pi i s\tau}-1\right|	\leqslant2\pi |s|\tau$, we have
$$
|m_z(s)|\leqslant1+2\pi|az||s|\int_0^{\frac{1}{2}}\tau^{a\mathrm{Re} z}\md{\tau} \leqslant1+\pi|s||az|\leqslant1+|az|.
$$
When $\pi |s|>1$, we have
\begin{align*}
	m_z(s)&=2^{-az}+az\int_0^{\frac{1}{2\pi|s|}}\tau^{az-1}\left(e^{2\pi i s\tau}-1\right)\md{\tau}+az\int_{\frac{1}{2\pi|s|}}^{\frac{1}{2}}\tau^{az-1}\left(e^{2\pi i s\tau}-1\right)\md{\tau}\\
	&=\left(\frac{1}{2\pi|s|}\right)^{az}+az\int_0^{\frac{1}{2\pi|s|}}
	\tau^{az-1}\left(e^{2\pi i s\tau}-1\right)\md{\tau}+az\int_{\frac{1}{2\pi|s|}}^{\frac{1}{2}}
	\tau^{az-1}e^{2\pi i s\tau}\md{\tau}.
\end{align*}
The sum of the absolute values of the first two terms is bounded by
\begin{align*}
	\left(\frac{1}{2\pi|s|}\right)^{a\mathrm{Re} z}
	+2\pi|az||s|\int_0^{\frac{1}{2\pi|s|}}\tau^{a\mathrm{Re} z}\md{\tau}\leqslant1+2\pi \frac{|az||s|}{a\mathrm{Re} z+1}\left(\frac{1}{2\pi |s|}\right)^{a\mathrm{Re} z+1}\leqslant 1+|az|.
\end{align*}
For the last term, integration by parts gives
$$
az\int_{\frac{1}{2\pi|s|}}^{\frac{1}{2}}\tau^{az-1}e^{2\pi i s\tau}\md{\tau}
=\frac{az}{2\pi i s}
\left[\tau^{az-1}e^{2\pi i s\tau}\right]_{\frac{1}{2\pi|s|}}^{\frac{1}{2}}
-\frac{az(az-1)}{2\pi i s}\int_{\frac{1}{2\pi|s|}}^{\frac{1}{2}}\tau^{az-2}e^{2\pi i s\tau}\md{\tau}.
$$
Since $\mathrm{Re} z\geqslant 0$, $0<\tau\leqslant \frac{1}{2}$ and $\pi|s|>1$,
$$
|\tau^{az-1}|\leqslant\tau^{-1},\qquad|\tau^{az-2}|\leqslant\tau^{-2},\qquad 2+2\pi |s|\leqslant 4\pi|s|.
$$
Therefore
\begin{align*}
	\left|az\int_{\frac{1}{2\pi|s|}}^{\frac{1}{2}}
	\tau^{az-1}e^{2\pi i s\tau}\md{\tau}\right|
	&\leqslant \frac{|az|}{2\pi |s|}(2+2\pi|s|)+\frac{|az||az-1|}{2\pi|s|}\int_{\frac{1}{2\pi|s|}}^{\frac{1}{2}}
	\tau^{-2}\md{\tau}\\
	&\leqslant 2|az|+|az||az-1|\leqslant 3|az|+|az|^2,
\end{align*}
where we used $|az-1|\leqslant 1+|az|$ in the last inequality. Combining these two estimates, we obtain
$$
|m_z(s)|\leqslant1+|az|+3|az|+|az|^2\leqslant2(1+|az|)^2,
$$
which is the claim.

Since $\mathcal F_s$ is bounded on $L^2$, it commutes
with the Bochner integral defining $T_z$. For $\mathrm{Re} z>0$, applying the unitary map $\Fs$ to $T_z(f,g)$ and using
\eqref{eq:phase}, we obtain
\begin{equation*}
	\Fs T_z(f,g)(y,\eta)
	=e^{-\pi i y\cdot\eta}
	m_z(y\cdot\eta)A(f,g)(y,\eta).
\end{equation*}
For $\mathrm{Re} z=0$, define $T_z(f,g)$ by
$$
\mathcal F_sT_z(f,g)(y,\eta):=e^{-\pi i y\cdot\eta}m_z(y\cdot\eta)A(f,g)(y,\eta).
$$
Using the above estimate for $|m_z(s)|$ and \eqref{eq:moyal}, we obtain, for every $z$ with $\mathrm{Re} z\geqslant 0$,
\begin{align*}
	\|T_z(f,g)\|_2^2&=\int_{\R^{2d}}|m_z(y\cdot\eta)|^2|A(f,g)(y,\eta)|^2\md{y}\md{\eta}\\
	&\leqslant4(1+|az|)^4\|A(f,g)\|_2^2=4(1+|az|)^4\|f\|_2^2\|g\|_2^2.
\end{align*}
Hence
$$
\|T_z(f,g)\|_2\leqslant2(1+|az|)^2\|f\|_2\|g\|_2,\qquad\mathrm{Re} z\geqslant0.
$$

Moreover, since $m_0(s)=1$,
$$
\mathcal F_sT_0(f,g)(y,\eta)=e^{-\pi i y\cdot\eta}A(f,g)(y,\eta)=\mathcal F_sW_0(f,g)(y,\eta)
$$
and therefore
$$
T_0(f,g)=W_0(f,g).
$$

We now prove that $T_z$ is continuous on $\mathrm{Re} z\geqslant 0$. Fix $z_0$ with $\mathrm{Re}z_0\geqslant0$. Suppose that $z_n\to z_0$ with
$\mathrm{Re}z_n\geqslant0$ and $|z_n-z_0|<1$. Then
\begin{align*}
	\|T_{z_n}(f,g)-T_{z_0}(f,g)\|_2^2&=\int_{\R^{2d}}|m_{z_n}(y\cdot\eta)-m_{z_0}(y\cdot\eta)|^2\times
	|A(f,g)(y,\eta)|^2\md{y}\md{\eta}.
\end{align*}
For every fixed $s\in\R$, $m_{z_n}(s)\to m_{z_0}(s)$ as $n\to \infty$ and the integrand is dominated by
$$
16\bigl(1+a(|z_0|+1)\bigr)^4|A(f,g)(y,\eta)|^2.
$$
The continuity then follows from the dominated convergence theorem.
\end{proof}

\subsection{Completion of the proof}

We first obtain the right boundary estimate for $T_z$ from
Theorem \ref{thm:weighted}. Recall that $a=\frac{3d}{8}$.
For $t\in\R$, set
\begin{equation*}
	v(\tau):=a(1+it)\tau^{a-1+iat}.
\end{equation*}
Then $|v(\tau)|\leqslant a(1+|t|)\tau^{a-1}$.
For $f,g\in L^2(\R^d)$, the definition of $T_z$
(see \eqref{eq:half-family}) gives
\begin{equation*}
	T_{1+it}(f,g)
	=a(1+it)\int_0^{\frac{1}{2}}
	\tau^{a-1+iat}W_\tau(f,g)\md{\tau}
	=Q_v(f,g).
\end{equation*}
Applying Theorem \ref{thm:weighted} with $q=8$ and
$\alpha=d(\frac{1}{2}-\frac{1}{8})=a$ gives
\begin{equation}\label{eq:half-right-bound}
	\|T_{1+it}(f,g)\|_8
	\leqslant C_d(1+|t|)\|f\|_{2,8}\|g\|_{2,8}.
\end{equation}

\begin{proof}[Proof of Theorem \ref{thm:main}]
	Fix $f,g\in\mcS(\R^d)$. By Proposition \ref{prop:imaginary},
	the map $z\mapsto T_z(f,g)$ is strongly holomorphic in the open
	strip, continuous to the boundary and of polynomial growth.
	Moreover, the boundary estimates \eqref{eq:half-L2-bound} and
	\eqref{eq:half-right-bound} verify the hypotheses of
	Lemma \ref{lem:stein} with $q=8$, $N_0=2$, $N_1=1$, and
	$A_0,A_1$ depending only on $d$.
	At the interpolation point $z=\theta=\frac{8}{3d}\in(0,1)$,
	\begin{equation*}
		\frac{1-\theta}{2}+\frac{\theta}{8}
		=\frac{1}{2}-\frac{3\theta}{8}
		=\frac{1}{2}-\frac{1}{d}
		=\frac{1}{p^*}.
	\end{equation*}
	Lemma \ref{lem:stein} therefore yields
	\begin{equation*}
		\|T_\theta(f,g)\|_{p^*}
		\leqslant C_d\|f\|_{2,p^*}\|g\|_{2,p^*}.
	\end{equation*}
	Combining this and \eqref{eq:half-reconstruction}, we obtain
	\begin{align*}
		\|W_{BJ}(f,g)\|_{p^*}\leqslant\|T_\theta(f,g)\|_{p^*}
		+\|T_\theta(g,f)\|_{p^*}\leqslant C_d\|f\|_{2,p^*}\|g\|_{2,p^*}.
	\end{align*}
	This proves \eqref{eq:lorentz-endpoint} for Schwartz functions. 
	
	By density, \eqref{eq:lorentz-endpoint} gives a unique continuous
    sesquilinear extension of $W_{BJ}$ to
    $L^{2,p^*}(\R^d)\times L^{2,p^*}(\R^d)$, agreeing with the original
    $L^2$-valued average, and the sharpness follows from Proposition \ref{prop:lorentz-sharpness}.
\end{proof}

\begin{proof}[Proof of Corollary \ref{cor:bj-operator}]
	Let $a\in L^{q^*}(\R^{2d})$. For $f,g\in\mcS(\R^d)$,
	applying H\"older's inequality and Theorem \ref{thm:main}, we obtain
	\begin{align*}
		|\langle\OpBJ(a)f,g\rangle|=\left|\int_{\R^{2d}}
		a(X)\overline{W_{BJ}(g,f)(X)}\md{X}\right|\leqslant\|a\|_{q^*}\|W_{BJ}(g,f)\|_{p^*}\leqslant C_d\|a\|_{q^*}\|f\|_2\|g\|_2.
	\end{align*}
	Thus the defining sesquilinear form extends to a bounded form on
	$L^2(\R^d)\times L^2(\R^d)$. The Riesz representation theorem
	gives the asserted operator and norm bound, and uniqueness
	follows from density.
\end{proof}

\end{document}